\documentclass[12pt,a4paper,twoside]{article}
\usepackage{graphics}
\usepackage{amssymb}
\usepackage{amsmath}
\usepackage{mathrsfs}
\usepackage{makeidx}
\usepackage{color}
\usepackage{graphicx}
\usepackage{subfigure}
\usepackage{multicol}
\usepackage{multirow}
\usepackage{float}
\usepackage{booktabs}
\usepackage{epsfig}
\usepackage{multirow}
\usepackage{epstopdf}
\usepackage[colorlinks,
linkcolor=red,
anchorcolor=blue,
citecolor=blue
]{hyperref} 
\usepackage{caption}
\usepackage{cases}
\usepackage{algorithm,algorithmic}

\usepackage{latexsym, cite, bm, amsthm, amsmath}
\usepackage{enumerate}
\usepackage{marginnote}
\usepackage{epstopdf}
\usepackage{pifont}

\def \[{\begin{equation}}
	\def \]{\end{equation}}

\newtheorem{thm}{Theorem}[section]

\newtheorem{lem}{Lemma}[section]

\newtheorem{exam}{Example}[section]

\numberwithin{equation}{section}

\title{\bf A monotone block coordinate descent method for solving absolute value equations}

\usepackage[marginal]{footmisc}
\usepackage{authblk}
\author[a]{Tingting Luo\thanks{Email address: 610592494@qq.com.}}
\author[a]{Jiayu Liu\thanks{Email address: 1977078576@qq.com.}}
\author[a]{Cairong Chen\thanks{Supported in part by the Fujian Alliance of Mathematics (2023SXLMQN03). Email address: cairongchen@fjnu.edu.cn.}}
\author[b]{Qun Wang\thanks{Corresponding author. Supported in part by the National Natural Science Foundation of China (11801502) and the Zhejiang Provincial Philosophy and Social Sciences Planning Project (24NDJC113YB). Email address:  wangqun876@gmail.com.}}
\affil[a]{School of Mathematics and Statistics \& Key Laboratory of Analytical Mathematics and Applications (Ministry of Education) \& Fujian Key Laboratory of Analytical Mathematics and Applications (FJKLAMA) \& Center for Applied Mathematics of Fujian Province (FJNU), Fujian Normal University, Fuzhou, 350117,  P.R. China.}
\affil[b]{School of Data Sciences, Zhejiang University of Finance and
Economics, Hangzhou, 310018, P.R. China.}

\begin{document}
\date{\today}
\maketitle

\begin{quote}
{\bf Abstract:} In this paper, we proposed a monotone block coordinate descent method for solving absolute value equation (AVE). Under appropriate conditions, we analyzed the global convergence of the algorithm and conduct numerical experiments to demonstrate its feasibility and effectiveness.

\end{quote}

\section{Introduction}\label{sec:intro}

We consider the system of absolute value equations (AVE) of the type
\begin{equation}\label{eq:ave}
Ax - \vert x \vert = b,
\end{equation}
where~$A\in\mathbb{R}^{n\times n}$ and $b\in\mathbb{R}^n$ are known, and~$x\in\mathbb{R}^n$ is unknown. Here, $\vert x\vert=[\vert x_1\vert,\vert x_2\vert,\cdots,\vert x_n\vert]^\top$. Obviously, AVE~\eqref{eq:ave} is a special case of the following generalized absolute value equations (GAVE)
\begin{equation}\label{eq:gave}
Ax + B\vert x \vert = b,
\end{equation}
where~$B \in \mathbb{R}^{n\times n}$ is given. To the best of our knowledge, GAVE~\eqref{eq:gave} was first introduced in~\cite{rohn2004}. As demonstrated in~\cite{mang2007}, solving the general GAVE~\eqref{eq:gave} is NP-hard. Furthermore, if it is solvable, checking whether GAVE~\eqref{eq:gave} has a unique solution or multiple solutions is NP-complete~\cite{prok2009}.

AVE \eqref{eq:ave} and GAVE \eqref{eq:gave} are significant mathematical models with practical applications in engineering and information security. For instance, a transform function based on GAVE \eqref{eq:gave} has been used to effectively improve the security of cancellable biometric systems \cite{dnhl2023}. Meanwhile, AVE \eqref{eq:ave} and GAVE \eqref{eq:gave} are equivalent to the linear complementarity problem \cite{cops1992,huhu2010,mang2007,mame2006,prok2009}, and are closely related to the solution of linear interval equations \cite{rohn1989}. They have garnered increasing attention in the fields of numerical optimization and numerical algebra in recent years.

There are numerous results on both the theoretical and numerical aspects of AVE \eqref{eq:ave} and GAVE \eqref{eq:gave}. On the theoretical side, conditions for the existence, nonexistence, and uniqueness of solutions to AVE \eqref{eq:ave} and GAVE \eqref{eq:gave} have been reported; see, e.g., \cite{hlad2018,mang2007,mame2006,mezz2020,rohn2004,rohn2009,rohf2014,wuli2018,wush2021} and references therein. On the numerical side, the discrete iterative methods \cite{chyh2024,chyh2023,doss2020,edhs2017,ghlw2017,keyf2020,kema2017,mang2009,nika2011,rohf2014,wacc2019,yuch2022,yuly2021,zhwl2021} and continuous dynamic models \cite{cyyh2021,gawa2014,jyfc2023,lyyh2023,maer2018,sanc2019} for solving AVE \eqref{eq:ave} and GAVE \eqref{eq:gave} have been studied. Notably, Noor, Iqbal, Khattri and AI-Said proposed an iterative method for solving AVE \eqref{eq:ave} based on the minimization technique \cite{nika2011}, which can be viewed as a modification of the Gauss-Seidel approach. Specifically, the solution to AVE \eqref{eq:ave} is reformulated by finding the minimum point of an unconstrained optimization problem. Subsequently, a minimization technique is used to search along each block coordinate direction to solve the corresponding unconstrained optimization problem. However, it is important to note that the objective function of the constructed unconstrained optimization problem is not globally quadratic differentiable. Consequently, the approach in \cite{nika2011} misused the second-order Taylor expansion of the objective function to determine the step sizes, which might make the method being non-monotonic descent and the rigorously  global convergence analysis is lack in \cite{nika2011}. Therefore, after fully exploring the intrinsic properties of the objective function proposed in \cite{nika2011}, this paper develops a monotone block coordinate descent method with guaranteed convergence to solve the AVE \eqref{eq:ave}.

The rest of this paper is organized as follows. In Section \ref{sec:Pre}, we present some lemmas which are useful for our later developments. In Section \ref{sec:analysis}, a monotone block coordinate descent algorithm for solving AVE \eqref{eq:ave} is proposed and its convergence is analyzed. The numerical experiments and the conclusion of this paper are given in Section \ref{sec:Num} and Section \ref{sec:conclusion}, respectively.

\textbf{Notation.}
Let $\mathbb{R}^{m\times n}$ be the set of all $m\times n$ real matrices and $\mathbb{R}^m=\mathbb{R}^{m\times 1}$. $I$ denotes the identity
matrix with suitable dimensions, $I(:,i)$ represents the $i$ column of the identity matrix $I$. For the vector $x\in\mathbb{R}^n$, $x_i$ represents its $i$th element. For vectors $x,y\in\mathbb{R}^n$, $\langle x,y\rangle=x^\top y$, where $\cdot ^\top$ denotes the transpose operation. $A=(a_{ij})\in\mathbb{R}^{m\times n}$ indicates that the element in the row $i$ and the column $j$ of $A$ is $a_{ij}$. For $x\in\mathbb{R}$, the symbolic function is defined as
\begin{equation*}
{\rm sign}(x)=
\begin{cases}
1,~~\text{if}~x>0;\\
0,~~\text{if}~x=0;\\
-1,~~\text{if}~x<0.
\end{cases}
\end{equation*}
For $x\in\mathbb{R}^n$, ${\rm sign}(x)$ represents the vector formed by taking the sign function value of the corresponding element, $\mathcal{D}(x)={\rm diag}({\rm sign}(x))$ denotes a diagonal matrix with a sign function vector as its diagonal elements.

\section{Preliminaries}\label{sec:Pre}
In this section, we provide the preparatory knowledge required for the later development.

Let $A\in\mathbb{R}^{n\times n}$ be a symmetric matrix and $b\in\mathbb{R}^n$. Define
\begin{equation}\label{eq:ff}
f(x)=\langle Ax,x\rangle - \langle |x|,x\rangle - 2\langle b,x\rangle,\quad x\in\mathbb{R}^n.
\end{equation}
As shown in \cite{nika2011}, the function $f$ defined as in \eqref{eq:ff} is continuously differentiable and
\begin{equation}\label{eq:diff}
\nabla f(x)=2(Ax-|x|-b).
\end{equation}

In the following, we will explore properties of the function $f$ defined as in \eqref{eq:ff}.

\begin{lem}\label{lem:1}
The function $\nabla f$ defined as in \eqref{eq:diff} is Lipschitz continuous on $\mathbb{R}^n$ with the Lipschitz constant $2(\|A\|+1)$.
\end{lem}
\begin{proof}
For any $x,y\in\mathbb{R}^n$, it follows from \eqref{eq:diff} that
\begin{align*}
\|\nabla f(x)-\nabla f(y)\|&=\|2(Ax-|x|-b)-2(Ay-|y|-b)\|\\
&=\|2A(x-y)-2(|x|-|y|)\|\\
&\leq 2(\|A\|+1)\|x-y\|,
\end{align*}
where the last inequality uses the triangle inequality and $\||x|-|y|\| \leq \|x-y\|.$
\end{proof}

\begin{lem}\label{lem:2}
If $A-I$ is a symmetric positive definite matrix, then the function $f$ defined by \eqref{eq:ff} is strong convex on $\mathbb{R}^n$.
\end{lem}
\begin{proof}
To prove that $f$ is strong convex on $\mathbb{R}^n$, it suffices to show that $\nabla f$ is strongly monotonic on $\mathbb{R}^n$. For any $x,y\in\mathbb{R}^n$ with $x\neq y$, we have
\begin{align*}
\langle \nabla f(x)-\nabla f(y),x-y\rangle
&=\langle 2A(x-y)-2(|x|-|y|),x-y \rangle\\
&=2\langle A(x-y),x-y \rangle-2\langle |x|-|y|,x-y\rangle\\
&\geq 2\langle A(x-y),x-y \rangle-2\langle x-y,x-y\rangle\\
&=2(x-y)^\top (A-I)(x-y)\\
&\ge 2\lambda_{\min}(A-I)\|x-y\|^2
\end{align*}
where the first inequality uses $\||x|-|y|\| \leq \|x-y\|$. The proof is completed since the smallest eigenvalue $\lambda_{\min}(A-I) >0$.
\end{proof}

\begin{thm}\label{thm:1}
If $A-I$ is a symmetric positive definite matrix, then for any $b \in \mathbb{R}^n$, the function $f$ defined as in \eqref{eq:ff} has a unique minimum point, which is the unique solution of AVE~\eqref{eq:ave}.
\end{thm}
\begin{proof}
Since $A-I$ is a symmetric positive definite matrix, then the interval matrix $[A-I,A+I]$ is regular. Therefore, for any $b \in \mathbb{R}^n$, AVE \eqref{eq:ave} has a unique solution $x_*$ \cite{wuli2018}. From \eqref{eq:diff}, we know $\nabla f(x_*) = 0$. In addition, it follows from Lemma \ref{lem:2} that the differentiable function $f$ is  strong convex on $\mathbb{R}^n$. Hence, $x_*$ is the unique minimum point of the function $f$.
\end{proof}

\section{Monotone block coordinate descent method}\label{sec:analysis}

Before proposing the monotone block coordinate descent method for solving AVE \eqref{eq:ave}, we first review the iterative method presented in \cite{nika2011}. In the method proposed in \cite{nika2011} (see Algorithm \ref{alg:1} for details), the Taylor expansion of $f(y^{(i)}+\alpha v^{(i)}+\beta v^{(j)})$ at $y^{(i)}$ is minimized to obtain $\alpha^{(i)}$ and $\beta^{(i)}$. In other words, $\alpha^{(i)}$ and $\beta^{(i)}$ are determined by minimizing \begin{equation}\label{eq:taylor}
f(y^{(i)}+\alpha v^{(i)}+\beta v^{(j)})=f(y^{(i)})+a\alpha^2+d\beta^2+2p^{(i)}\alpha+2p^{(j)}\beta+2c\alpha\beta
\end{equation}
in terms of $\alpha$ and $\beta$.  However, since the function $f$ is not actually globally second-order continuously differentiable, the Taylor expansion \eqref{eq:taylor} dose not hold on $\mathbb{R}^n$, which may make the iterative sequence $\{x^{(k)}\}$ generated by Algorithm~\ref{alg:1} satisfying $f(x^{(k+1)}) > f(x^{(k)})$ for some $k$. The following Example~\ref{exam:1} is used to demonstrate this phenomenon.

\begin{algorithm}[t]
\caption{The iterative method proposed in \cite{nika2011}.}\label{alg:1}
\begin{algorithmic}[1]
    \REQUIRE a symmetric matrix $A\in\mathbb{R}^{n\times n}$ and a vector $b\in\mathbb{R}^n$. Take the initial vector $x^{(0)}\in\mathbb{R}^n$. Set $k:=0$.

    \ENSURE an approximate solution of AVE \eqref{eq:ave}.

    \REPEAT
       \STATE For $i=1,2,\ldots,n$,
  \begin{align}
  &j=i-1,\nonumber\\
  &if~i=1,~then~j=n,\nonumber\\
  &v^{(i)}=I(:,i),~~v^{(j)}=I(:,j),\nonumber\\
  &\text{if}~i=1,~\text{then}~y^{(i)}=x^{(k)};\nonumber\\
  &C=A-\mathcal{D}(y^{(i)}),\nonumber\\
  &a=\langle Cv^{(i)},v^{(i)}\rangle,~~d=\langle Cv^{(j)},v^{(j)}\rangle,\nonumber\\
  &c=\langle Cv^{(i)},v^{(j)}\rangle=\langle Cv^{(j)},v^{(i)}\rangle,\nonumber\\
  &p^{(i)}=\langle Ay^{(i)}-|y^{(i)}|-b,v^{(i)}\rangle,~p^{(j)}=\langle Ay^{(i)}-|y^{(i)}|-b,v^{(j)}\rangle,\nonumber\\
  &\alpha^{(i)}=\frac{cp^{(j)}-dp^{(i)}}{ad-c^2},~\beta^{(i)}=\frac{cp^{(i)}-ap^{(j)}}{ad-c^2},\label{eq:step}\\
  &y^{(i+1)}=y^{(i)}+\alpha^{(i)}v^{(i)}+\beta^{(i)}v^{(j)}.\label{eq:yi}
  \end{align}

   \STATE Let $x^{(k+1)}=y^{(n+1)}, k = k+1$.

    \UNTIL{convergence};
    \RETURN the last $x^{(k+1)}$ as the approximation to the solution of AVE~\eqref{eq:ave}.
\end{algorithmic}
\end{algorithm}

\begin{exam}\label{exam:1}{\rm 
Consider AVE \eqref{eq:ave} with
\begin{equation*}
A=\begin{bmatrix}
\frac{3}{2}&\frac{1}{4}\\
\frac{1}{4}&\frac{3}{2}
\end{bmatrix}\quad \text{and}\quad b =\begin{bmatrix}
\frac{1}{4}\\
1
\end{bmatrix}.
\end{equation*}
Then $A-I$ is symmetric positive definite and the unique solution of this AVE is $x_*=[-\frac{2}{19},\frac{39}{19}]^\top$. In addition,
\begin{align}\label{eq:f'}
f(x)
=&\frac{3}{2}x_1^2+\frac{3}{2}x_2^2+\frac{1}{2}x_1x_2-x_1|x_1|-x_2|x_2|-\frac{1}{2}x_1-2x_2\nonumber\\
=&\begin{cases}
\frac{1}{2}x_1^2+\frac{1}{2}x_2^2+\frac{1}{2}x_1x_2-\frac{1}{2}x_1-2x_2 = f_1(x),~~\text{if}~x_1\geq0~\text{and}~x_2\geq0;\\
\frac{1}{2}x_1^2+\frac{5}{2}x_2^2+\frac{1}{2}x_1x_2-\frac{1}{2}x_1-2x_2=f_2(x),~~\text{if}~x_1\geq0~\text{and}~x_2<0;\\
\frac{5}{2}x_1^2+\frac{1}{2}x_2^2+\frac{1}{2}x_1x_2-\frac{1}{2}x_1-2x_2=f_3(x),~~\text{if}~x_1<0~\text{and}~x_2\geq0;\\
\frac{5}{2}x_1^2+\frac{5}{2}x_2^2+\frac{1}{2}x_1x_2-\frac{1}{2}x_1-2x_2=f_4(x),~~\text{if}~x_1<0~\text{and}~x_2<0.
\end{cases}
\end{align}

According to Algorithm \ref{alg:1}, $v^{(1)}=[1,0]^\top, v^{(2)}=[0,1]^\top$. The numerical results obtained by selecting different initial values of $x^{(0)}$ are shown in Table 1. As can be seen from Table \ref{table1}, when $x^{(0)}=[0.6,1.2]^\top$, we have $f(y^{(1)})<f(y^{(2)})$. The reason for this is that for the initial point $x^{(0)}=[0.6,1.2]^\top$, $\alpha^{(1)}=-\frac{19}{15}$ and $\beta^{(1)}=\frac{17}{15}$ are obtained by minimizing $g(\alpha,\beta)=f_1(y^{(1)})+\langle f_1^\prime(y^{(1)}),\alpha v^{(1)}+\beta v^{(2)} \rangle +\frac{1}{2}\langle f_1^{\prime\prime}(y^{(1)})(\alpha v^{(1)}+\beta v^{(2)}),\alpha v^{(1)}+\beta v^{(2)}\rangle$, which is the Taylor expansion of $f_1(y^{(1)}+\alpha v^{(1)}+\beta v^{(2)})$ at $y^{(1)}$. Then we have $f(y^{(1)}) =f_1(y^{(1)}) < f_1(y^{(2)})$. However, since $y^{(2)}_1 = -\frac{2}{3}< 0$, it follows from \eqref{eq:f'} that $f(y^{(2)})\neq f_1(y^{(2)})$. Indeed, we have $-\frac{23}{18}=f(y^{(2)})> f_1(y^{(2)})=-\frac{39}{18}> f(y^{(1)})$. In conclusion, since the minimum point of $f_1$ in $\mathbb{R}^2$ is different from the minimum point of function $f$ in $\mathbb{R}^2$, to minimize $f_1$ in $\mathbb{R}^2$ does not guarantee a decrease in the value of function $f$. Conversely, since $f_3$ and $f$ have the same minimum point in $\mathbb{R}^2$, minimizing $f_3$ on $\mathbb{R}^2$ is equivalent to minimizing $f$. This is why in Table \ref{table1} the Algorithm \ref{alg:1} can get the solution $x_*$ by iterating one step from $x^{(0)}=[-0.6,1.2]^\top$ or $y^{(2)} = [-\frac{2}{3}, \frac{7}{3}]$.

\renewcommand{\arraystretch}{1.5}
\begin{table}[H]
\centering
\caption{Numerical results of Example~\ref{exam:1}.}\label{table1}
\begin{tabular}{|c|c|c|c|c|c|c}\hline
$x^{(0)}=y^{(1)}$  &$f(y^{(1)})$   &$y^{(2)}$   &$f(y^{(2)})$    &$y^{(3)}$  &$f(y^{(3)})$   \\\hline
$[0.6,1.2]^\top$    &$-\frac{36}{25}$   &$[-\frac{2}{3},\frac{7}{3}]^\top$   &$-\frac{23}{18}$    &$[-\frac{2}{19},\frac{39}{19}]^\top$    &$-\frac{77}{38}$    \\\hline
$[-0.6,1.2]^\top$     &$-\frac{21}{25}$   &$[-\frac{2}{19},\frac{39}{19}]^\top$    &$-\frac{77}{38}$   &    &  \\\hline
\end{tabular}
\end{table}

}
\end{exam}

As demonstrated in Example \ref{exam:1}, the value of the objective function may increase during the iterations of Algorithm \ref{alg:1}, which makes Algorithm \ref{alg:1} non-monotonic and lacking a strict convergence analysis. This motivates us to propose a monotone block coordinate descent method which has a convergence guarantee. Our method is described in Algorithm \ref{alg:2}.

\begin{algorithm}[t]
\caption{Monotone block coordinate descent method for solving AVE \eqref{eq:ave}.}\label{alg:2}
\begin{algorithmic}[1]
    \REQUIRE a symmetric matrix $A\in\mathbb{R}^{n\times n}$ and the vector $b\in\mathbb{R}^n$. Take the initial vector $x^{(0)}\in\mathbb{R}^n$.

    \ENSURE an approximate solution of AVE \eqref{eq:ave}.

    \REPEAT
       \STATE Iteration $k + 1$ with $k\ge 0$. Given $x^{(k)}=[(x_{n_1}^{(k)})^\top, (x_{n_2}^{(k)})^\top,\cdots, (x_{n_N}^{(k)})^\top]^\top \in \mathbb{R}^{n}$ with $n_1 + n_2 + \cdots + n_N = n$, choose $s=r$ at iterations $r,r+N,r+2N,\ldots,$ for $r=1,\ldots,N$ and compute a new iterate $x^{(k+1)}$ satisfying
      \begin{align} \label{eq:alpha}
     x_{n_s}^{(k+1)} &= {\rm arg}\min_{x_{n_s}} f(x_{n_1}^{(k)}, \cdots, x_{n_{s-1}}^{(k)}, x_{n_s}, x_{n_{s+1}}^{(k)},\cdots,x_{n_{N}}^{(k)}),\\\nonumber
     x_{n_j}^{(k+1)} &= x_{n_j}^{(k)}, \quad \forall j\neq s.
      \end{align}

    \UNTIL{convergence};
    \RETURN the last $x^{(k+1)}$ as the approximation to the solution of AVE~\eqref{eq:ave}.
\end{algorithmic}
\end{algorithm}

In the following, we will discuss how to solve \eqref{eq:alpha}. For $s = r$ with $r\in \{1,2,\cdots,N\}$, let $v^{(i)} = [1,0]^\top$ and $v^{(j)} =[0,1]^\top$. Then we have $x_{n_s}^{(k+1)} = x_{n_s}^{(k)} + \alpha^{(k)}v^{(i)} + \beta^{(k)}v^{(j)}$, in which
\begin{equation}\label{eq:step}
[ \alpha^{(k)},\beta^{(k)}] = {\rm arg}\min_{\alpha,\beta} f(x^{(k)} + \alpha I(:,2s - 1)+\beta I(:,2s)).
\end{equation}
In order to solve \eqref{eq:step}, we conduct some properties of $f$.

Inspired by Example \ref{exam:1}, for
\begin{align}\label{eq:f}
f(x)
&=\begin{bmatrix}
x_1 & x_2
\end{bmatrix}
\begin{bmatrix}
a_{11} & a_{12}\\
a_{12} & a_{22}
\end{bmatrix}
\begin{bmatrix}
x_1 \\ x_2
\end{bmatrix}
-\begin{bmatrix}
x_1 & x_2
\end{bmatrix}
\begin{bmatrix}
|x_1| \\ |x_2|
\end{bmatrix}
-2\begin{bmatrix}
b_1 & b_2
\end{bmatrix}
\begin{bmatrix}
x_1 \\ x_2
\end{bmatrix}\nonumber\\
&=a_{11} x_1^2+a_{22} x_2^2+2a_{12}x_1x_2-x_1|x_1|-x_2|x_2|-2b_1x_1-2b_2x_2\nonumber\\
&=\begin{cases}
(a_{11}-1)x_1^2+(a_{22}-1)x_2^2+2a_{12}x_1x_2-2b_1x_1-2b_2x_2=f_1(x),~\text{if}~x_1\ge 0~\text{and}~x_2\ge 0;\\
(a_{11}-1)x_1^2+(a_{22}+1)x_2^2+2a_{12}x_1x_2-2b_1x_1-2b_2x_2=f_2(x),~\text{if}~x_1\ge 0~\text{and}~x_2<0;\\
(a_{11}+1)x_1^2+(a_{22}-1)x_2^2+2a_{12}x_1x_2-2b_1x_1-2b_2x_2=f_3(x),
~\text{if}~x_1<0~\text{and}~x_2\ge 0;\\
(a_{11}+1)x_1^2+(a_{22}+1)x_2^2+2a_{12}x_1x_2-2b_1x_1-2b_2x_2=f_4(x),
~\text{if}~x_1<0~\text{and}~x_2<0,
\end{cases}
\end{align}
we present the following lemma.

\begin{lem}\label{lem:3}
Let $f$ and $f_i(i = 1,2,3,4)$ be defined as in \eqref{eq:f}. If $A-I\in \mathbb{R}^{2\times 2}$ is a symmetric positive definite matrix, then for any $b\in \mathbb{R}^2$ we have
\begin{itemize}
  \item [(1)] If we extend the domain of $f_1$ to $\mathbb{R}^2$, then $f_1$ has the unique minimum point
      $$p=\left[\frac{b_2a_{12}-b_1(a_{22}-1)}{a_{12}^2-(a_{11}-1)(a_{22}-1)},
\frac{b_1a_{12}-b_2(a_{11}-1)}{a_{12}^2-(a_{11}-1)(a_{22}-1)}\right]^\top;$$

  \item [(2)] If we extend the domain of $f_2$ to $\mathbb{R}^2$, then $f_2$ has the unique minimum point
      $$q =\left[\frac{b_2a_{12}-b_1(a_{22}+1)}{a_{12}^2-(a_{11}-1)(a_{22}+1)},
\frac{b_1a_{12}-b_2(a_{11}-1)}{a_{12}^2-(a_{11}-1)(a_{22}+1)}\right]^\top;$$

  \item [(3)] If we extend the domain of $f_3$ to $\mathbb{R}^2$, then $f_3$ has the unique minimum point
      $$r =\left[\frac{b_2a_{12}-b_1(a_{22}-1)}{a_{12}^2-(a_{11}+1)(a_{22}-1)},
\frac{b_1a_{12}-b_2(a_{11}+1)}{a_{12}^2-(a_{11}+1)(a_{22}-1)}\right]^\top;$$

  \item [(4)] If we extend the domain of $f_4$ to $\mathbb{R}^2$, then $f_4$ has the unique minimum point
      $$s = \left[\frac{b_2a_{12}-b_1(a_{22}+1)}{a_{12}^2-(a_{11}+1)(a_{22}+1)},
\frac{b_1a_{12}-b_2(a_{11}+1)}{a_{12}^2-(a_{11}+1)(a_{22}+1)}\right]^\top;$$

\item[(5)] One of $p$, $q$, $r$ and $s$ is the unique minimum point of $f$ in $\mathbb{R}^2$.
\end{itemize}
\end{lem}

\begin{proof}
Since $A - I\in \mathbb{R}^2$ is a symmetric positive definite matrix, we have
\begin{equation}\label{eq:condition}
(a_{11}-1)(a_{22}-1)>a_{12}^2>0,\quad a_{11}-1>0,\quad a_{22}-1>0.
\end{equation}
For the function $f_1$ in $\mathbb{R}^2$, we have
\begin{align*}
\frac{\partial f_1}{\partial x_1}&=2(a_{11}-1)x_1+2a_{12}x_2-2b_1,\\
\frac{\partial f_1}{\partial x_2}&=2(a_{22}-1)x_2+2a_{12}x_1-2b_2,\\
\frac{\partial^2 f_1}{\partial x_1^2}&=2(a_{11}-1),\\
\frac{\partial^2 f_1}{\partial x_2^2}&=2(a_{22}-1),\\
\frac{\partial^2 f_1}{\partial x_1\partial x_2}&=2a_{12}.\\
\end{align*}
Then, for any $b\in \mathbb{R}^2$, it follows from \eqref{eq:condition} that $f_1$ has the unique minimum point $p$ in $\mathbb{R}^2$. Similarly, we can prove the results of (2)--(4).

In the following, we will prove the result in (5). Since $A-I$ is a symmetric positive definite matrix, for any $b\in \mathbb{R}^2$, Lemma \ref{lem:2} implies that $f$ has a unique minimum point in $\mathbb{R}^2$. In addition, since $f$ is convex, any local minimum point is the global minimum point. To complete the proof, we are required to verify that one of $p$, $q$, $r$ and $s$ is the minimum point of $f$ in $\mathbb{R}^2$. To this end, we first show that there is no $A\in \mathbb{R}^{2\times 2}$ and $b\in \mathbb{R}^2$ with $A-I$ being symmetric positive definite such that
\begin{align*}
&p_1 <0\quad \text{or} \quad p_2 <0,\\
&q_1 <0\quad \text{or} \quad q_2 \ge 0,\\
&r_1 \ge 0\quad \text{or} \quad r_2 < 0,\\
&s_1 \ge 0\quad \text{or} \quad s_2 \ge 0.
\end{align*}
This implies that at least one of  $p$, $q$, $r$ and $s$ is the minimum point of $f$ in $\mathbb{R}^2$.
Indeed, if $p_1<0$, then $r_1\ge 0$ is not true, which implies $r_2 <0$. It follows from $r_2 <0$ that $q_2\ge 0$ and $s_1\ge 0$. Then we have
\begin{align}\label{eq:1a}
b_2a_{12}-b_1(a_{22}-1)>0,\\\label{eq:4a}
b_2a_{12}-b_1(a_{22}+1)\leq0,\\\label{eq:2b}
b_1a_{12}-b_2(a_{11}-1)\leq0.\\\label{eq:3b}
b_1a_{12}-b_2(a_{11}+1)>0,
\end{align}
It follows from \eqref{eq:1a} and \eqref{eq:4a} that $b_1>0$ and \eqref{eq:2b} and \eqref{eq:3b} imply $b_2<0$. Substituting $b_1>0,b_2<0$ into \eqref{eq:1a} yields that $a_{12}<0$. Then it follows from \eqref{eq:1a} and \eqref{eq:2b} that $\frac{a_{12}}{a_{22}-1}<\frac{b_1}{b_2}$ and $\frac{a_{11}-1}{a_{12}}\geq\frac{b_1}{b_2}$, from which we have $\frac{a_{12}}{a_{22}-1}<\frac{b_1}{b_2}\leq\frac{a_{11}-1}{a_{12}}$. Hence, $a_{12}^2>(a_{11}-1)(a_{22}-1)$, which contradicts to the first inequality of \eqref{eq:condition}.

If $p_2 < 0$ is satisfied, we have $q_1 < 0$, $r_1\ge 0$ and $s_2\ge 0$. Namely,
\begin{align}\label{eq:1b}
b_1a_{12}-b_2(a_{11}-1)>0,\\\label{eq:4b}
b_1a_{12}-b_2(a_{11}+1)\leq0,\\\label{eq:2a}
b_2a_{12}-b_1(a_{22}+1)>0,\\\label{eq:3a}
b_2a_{12}-b_1(a_{22}-1)\leq0.
\end{align}
From \eqref{eq:1b} and \eqref{eq:4b}, we have $b_2>0$. In addition, it follows from \eqref{eq:2a} and \eqref{eq:3a}that $b_1<0$. Substituting $b_1<0$ and $b_2>0$ into \eqref{eq:1b} yields that $a_{12}<0$. Then it follows from \eqref{eq:3a} and \eqref{eq:1b} that $\frac{a_{12}}{a_{22}-1}\leq\frac{b_1}{b_2}$ and $\frac{a_{11}-1}{a_{12}}>\frac{b_1}{b_2}$, which imply $\frac{a_{12}}{a_{22}-1}\leq\frac{b_1}{b_2}<\frac{a_{11}-1}{a_{12}}$. Hence, $a_{12}^2>(a_{11}-1)(a_{22}-1)$, which also contradicts to the first inequality of \eqref{eq:condition}.

In the following, we will show that only one of the following four cases holds:
\begin{equation*}
\begin{array}{ll}
\text{[(i)]}~p\ge 0; &[(ii)]~q_1\ge 0~\text{and}~q_2<0;\\
\text{[(iii)]}~r_1<0~\text{and}~r_2\ge 0; & \text{[(iv)]}~ s<0.
\end{array}
\end{equation*}

When (i) holds,  (ii) and (iii) cannot be hold. Hence, we only need to prove that (iv) is also not satisfied. To this end, we only to prove that the following inequalities
\begin{align}\label{eq:1.1}
b_2a_{12}-b_1(a_{22}-1)\leq0,\\\label{eq:1.2}
b_1a_{12}-b_2(a_{11}-1)\leq0,\\\label{eq:4.1}
b_2a_{12}-b_1(a_{22}+1)>0,  \\ \label{eq:4.2}
b_1a_{12}-b_2(a_{11}+1)>0,
\end{align}
cannot hold simultaneously. Indeed, from \eqref{eq:1.1} and \eqref{eq:4.1}, we get $b_1<0$. In the same way, from \eqref{eq:1.2} and \eqref{eq:4.2} we have $b_2<0$. Substituting $b_1<0$ and $b_2<0$ into \eqref{eq:1.1} yields that $a_{12}>0$. Then it follows from  \eqref{eq:1.1} and \eqref{eq:1.2} that
$$\frac{a_{22}-1}{a_{12}} \leq \frac{b_2}{b_1} \leq \frac{a_{12}}{a_{11}-1},$$
from which we have $a_{12}^2>(a_{11}-1)(a_{22}-1)$, which contradicts to the first
inequality of~\eqref{eq:condition}.

If (ii) holds, (i) and (iv) cannot be satisfied. Hence,  we only need to prove that (iii) is also not satisfied. Namely, the inequalities
\begin{align} \label{eq:2.1}
b_2a_{12}-b_1(a_{22}+1)\leq0,\\\label{eq:2.2}
b_1a_{12}-b_2(a_{11}-1)>0, \\\label{eq:3.1}
b_2a_{12}-b_1(a_{22}-1)>0, \\\label{eq:3.2}
b_1a_{12}-b_2(a_{11}+1)\leq0,
\end{align}
do not have a solution. From \eqref{eq:2.1} and \eqref{eq:3.1}, we get $b_1>0$. In addition, it follows from \eqref{eq:2.2} and \eqref{eq:3.2} that $b_2>0$. Substituting $b_1>0$ and $b_2>0$ into \eqref{eq:2.2} yields that $a_{12}>0$. Then it follows from \eqref{eq:2.2} and \eqref{eq:3.1} that
$$\frac{a_{22}-1}{a_{12}} < \frac{b_2}{b_1} < \frac{a_{12}}{a_{11}-1},$$
that is, $a_{12}^2>(a_{11}-1)(a_{22}-1)$, which contradicts to the first inequality of \eqref{eq:condition}.

If (iii) holds, (i) and (iv) cannot hold. Moreover, we have proved that (ii) and (iii) cannot hold simultaneously.

If (iv) is satisfied, (ii) and (iii) cannot hold. In addition, we have proved that (i) and (iv) cannot hold simultaneously.
\end{proof}

For certain $s\in \{1,2,\cdots,N\}$, for simplicity, we denote $v^{(i)} = I(:,2s -1)$ and $v^{(j)} = I(:,2s)$ in the following. Let
\begin{align*}
g(\alpha,\beta)
&=f(x^{(k)}+\alpha v^{(i)}+\beta v^{(j)})\\
&=(x^{(k)}+\alpha v^{(i)}+\beta v^{(j)})^\top A (x^{(k)}+\alpha v^{(i)}+\beta v^{(j)})-(x^{(k)}+\alpha v^{(i)}+\beta v^{(j)})^\top |x^{(k)}+\alpha v^{(i)}+\beta v^{(j)}|\\
&-2b^\top(x^{(k)}+\alpha v^{(i)}+\beta v^{(j)})\\
&=\alpha^2 (v^{(i)})^\top Av^{(i)}+\beta^2 (v^{(j)})^\top Av^{(j)}+2\alpha (x^{(k)})^\top A v^{(i)}+2\beta (x^{(k)})^\top A v^{(j)} + 2\alpha \beta (v^{(i)})^\top Av^{(j)}\\
&\quad -(x^{(k)}+\alpha v^{(i)}+\beta v^{(j)})^\top \mathcal{D}( x^{(k)}+\alpha v^{(i)}+\beta v^{(j)})(x^{(k)}+\alpha v^{(i)}+\beta v^{(j)})\\
&\quad +(x^{(k)})^\top A x^{(k)}-2b^\top x^{(k)}-2\alpha b^\top v^{(i)}-2\beta b^\top v^{(j)}.
\end{align*}
Then we have
\begin{equation*}
g(\alpha,\beta)=\left\{
\begin{aligned}
&g_1(\alpha,\beta),~~\text{if}~x^{(k)}_i+\alpha \ge 0~\text{and}~x^{(k)}_j+\beta \ge 0;\\
&g_2(\alpha,\beta),~~\text{if}~x^{(k)}_i+\alpha \ge 0~\text{and}~x^{(k)}_j+\beta < 0;\\
&g_3(\alpha,\beta),~~\text{if}~x^{(k)}_i+\alpha < 0~\text{and}~x^{(k)}_j+\beta \ge 0;\\
&g_4(\alpha,\beta),~~\text{if}~x^{(k)}_i+\alpha < 0~\text{and}~x^{(k)}_j+\beta < 0,\\
\end{aligned}
\right.
\end{equation*}
where
\begin{align*}
g_1(\alpha,\beta)
&=(a_{ii}-1)\alpha^2+2\alpha[(x^{(k)})^\top Av^{(i)}-b^\top v^{(i)}-x^{(k)}_i]-(x^{(k)}_i)^2\\
&\quad +(a_{jj}-1)\beta^2+2\beta[(x^{(k)})^\top Av^{(j)}-b^\top v^{(j)}-x^{(k)}_j]-(x^{(k)}_j)^2\\
&\quad +2\alpha\beta a_{ij}+(x^{(k)})^\top Ax^{(k)}-2b^\top x^{(k)}-\sum\limits_{l\neq i,j}x^{(k)}_l |x^{(k)}_l|,\\
g_2(\alpha,\beta)
&=(a_{ii}-1)\alpha^2+2\alpha[(x^{(k)})^\top Av^{(i)}-b^\top v^{(i)}-x^{(k)}_i]-(x^{(k)}_i)^2\\
&\quad +(a_{jj}+1)\beta^2+2\beta[(x^{(k)})^\top Av^{(j)}-b^\top v^{(j)}+x^{(k)}_j]+(x^{(k)}_j)^2\\
&\quad +2\alpha\beta a_{ij}+(x^{(k)})^\top Ax^{(k)}-2b^\top x^{(k)}-\sum\limits_{l\neq i,j}x^{(k)}_l |x^{(k)}_l|,\\
g_3(\alpha,\beta)
&=(a_{ii}+1)\alpha^2+2\alpha[(x^{(k)})^\top Av^{(i)}-b^\top v^{(i)}+x^{(k)}_i]+(x^{(k)}_i)^2\\
&\quad +(a_{jj}-1)\beta^2+2\beta[(x^{(k)})^\top Av^{(j)}-b^\top v^{(j)}-x^{(k)}_j]-(x^{(k)}_j)^2\\
&\quad +2\alpha\beta a_{ij}+(x^{(k)})^\top Ax^{(k)}-2b^\top x^{(k)}-\sum\limits_{l\neq i,j}x^{(k)}_l |x^{(k)}_l|,\\
g_4(\alpha,\beta)
&=(a_{ii}+1)\alpha^2+2\alpha[(x^{(k)})^\top Av^{(i)}-b^\top v^{(i)}+x^{(k)}_i]+(x^{(k)}_i)^2\\
&\quad +(a_{jj}+1)\beta^2+2\beta[(x^{(k)})^\top Av^{(j)}-b^\top v^{(j)}+x^{(k)}_j]+(x^{(k)}_j)^2\\
&\quad +2\alpha\beta a_{ij}+(x^{(k)})^\top Ax^{(k)}-2b^\top x^{(k)}-\sum\limits_{l\neq i,j}x^{(k)}_l |x^{(k)}_l|.
\end{align*}
Since adding or subtracting a constant does not change the minimum point of a function, finding the minimum point of $g(\alpha,\beta)$ is equivalent to finding the minimum point of $\tilde{g}(\alpha,\beta)$, where
\begin{equation*}
\tilde{g}(\alpha,\beta)=\left\{
\begin{aligned}
&\tilde{g}_1(\alpha,\beta),~~\text{if}~x^{(k)}_i+\alpha \ge 0~\text{and}~x^{(k)}_j+\beta \ge 0,\\
&\tilde{g}_2(\alpha,\beta),~~\text{if}~x^{(k)}_i+\alpha \ge 0~\text{and}~x^{(k)}_j+\beta < 0,\\
&\tilde{g}_3(\alpha,\beta),~~\text{if}~x^{(k)}_i+\alpha < 0~\text{and}~x^{(k)}_j+\beta \ge 0,\\
&\tilde{g}_4(\alpha,\beta),~~\text{if}~x^{(k)}_i+\alpha < 0~\text{and}~x^{(k)}_j+\beta< 0,\\
\end{aligned}
\right.
\end{equation*}
in which
\begin{align*}
\tilde{g}_1(\alpha,\beta)
&=(a_{ii}-1)\alpha^2+2\alpha[(x^{(k)})^\top Av^{(i)}-b^\top v^{(i)}-x^{(k)}_i]-(x^{(k)}_i)^2\\
&\quad +(a_{jj}-1)\beta^2+2\beta[(x^{(k)})^\top Av^{(j)}-b^\top v^{(j)}-x^{(k)}_j]-(x^{(k)}_j)^2+2\alpha\beta a_{ij},\\
\tilde{g}_2(\alpha,\beta)
&=(a_{ii}-1)\alpha^2+2\alpha[(x^{(k)})^\top Av^{(i)}-b^\top v^{(i)}-x^{(k)}_i]-(x^{(k)}_i)^2\\
&\quad +(a_{jj}+1)\beta^2+2\beta[(x^{(k)})^\top Av^{(j)}-b^\top v^{(j)}+x^{(k)}_j]+(x^{(k)}_j)^2+2\alpha\beta a_{ij},\\
\tilde{g}_3(\alpha,\beta)
&=(a_{ii}+1)\alpha^2+2\alpha[(x^{(k)})^\top Av^{(i)}-b^\top v^{(i)}+x^{(k)}_i]+(x^{(k)}_i)^2\\
&\quad +(a_{jj}-1)\beta^2+2\beta[(x^{(k)})^\top Av^{(j)}-b^\top v^{(j)}-x^{(k)}_j]-(x^{(k)}_j)^2+2\alpha\beta a_{ij},\\
\tilde{g}_4(\alpha,\beta)
&=(a_{ii}+1)\alpha^2+2\alpha[(x^{(k)})^\top Av^{(i)}-b^\top v^{(i)}+x^{(k)}_i]+(x^{(k)}_i)^2\\
&\quad +(a_{jj}+1)\beta^2+2\beta[(x^{(k)})^\top Av^{(j)}-b^\top v^{(j)}+x^{(k)}_j]+(x^{(k)}_j)^2+2\alpha\beta a_{ij}.
\end{align*}

Let $t_1=x^{(k)}_i+\alpha,~t_2=x^{(k)}_j+\beta,~w_1=(v^{(i)})^\top(Ax^{(k)}-b),~w_2=(v^{(j)})^\top(Ax^{(k)}-b)$, then $\tilde{g}(\alpha,\beta)$ can be converted to $h(t_1,t_2)$ with
\begin{equation*}
h(t_1,t_2)=\left\{
\begin{aligned}
&h_1(t_1,t_2),~~\text{if}~t_1 \ge 0~\text{and}~t_2 \ge 0,\\
&h_2(t_1,t_2),~~\text{if}~t_1 \ge 0~\text{and}~t_2 < 0,\\
&h_3(t_1,t_2),~~\text{if}~t_1 < 0~\text{and}~t_2 \ge 0,\\
&h_4(t_1,t_2),~~\text{if}~t_1 < 0~\text{and}~t_2 < 0,\\
\end{aligned}
\right.
\end{equation*}
in which
\begin{align*}
h_1(t_1,t_2)
&=(a_{ii}-1)t_1^2-2(x^{(k)}_i a_{ii}-w_1+a_{ij}x^{(k)}_j)t_1+a_{ii}(x^{(k)}_i)^2-2w_1x^{(k)}_i\\
&\quad +(a_{jj}-1)t_2^2-2(x^{(k)}_j a_{jj}-w_2+a_{ij}x^{(k)}_i)t_2+a_{jj}(x^{(k)}_j)^2-2w_2x^{(k)}_j\\
&\quad +2a_{ij}(t_1t_2+x^{(k)}_i x^{(k)}_j),\\
h_2(t_1,t_2)
&=(a_{ii}-1)t_1^2-2(x^{(k)}_i a_{ii}-w_1+a_{ij}x^{(k)}_j)t_1+a_{ii}(x^{(k)}_i)^2-2w_1x^{(k)}_i\\
&\quad +(a_{jj}+1)t_2^2-2(x^{(k)}_j a_{jj}-w_2+a_{ij}x^{(k)}_i)t_2+a_{jj}(x^{(k)}_j)^2-2w_2x^{(k)}_j\\
&\quad +2a_{ij}(t_1t_2+x^{(k)}_i x^{(k)}_j),\\
h_3(t_1,t_2)
&=(a_{ii}+1)t_1^2-2(x^{(k)}_i a_{ii}-w_1+a_{ij}x^{(k)}_j)t_1+a_{ii}(x^{(k)}_i)^2-2w_1x^{(k)}_i\\
&\quad +(a_{jj}-1)t_2^2-2(x^{(k)}_j a_{jj}-w_2+a_{ij}x^{(k)}_i)t_2+a_{jj}(x^{(k)}_j)^2-2w_2x^{(k)}_j\\
&\quad +2a_{ij}(t_1t_2+x^{(k)}_i x^{(k)}_j),\\
h_4(t_1,t_2)
&=(a_{ii}+1)t_1^2-2(x^{(k)}_i a_{ii}-w_1+a_{ij}x^{(k)}_j)t_1+a_{ii}(x^{(k)}_i)^2-2w_1x^{(k)}_i\\
&\quad +(a_{jj}+1)t_2^2-2(x^{(k)}_j a_{jj}-w_2+a_{ij}x^{(k)}_i)t_2+a_{jj}(x^{(k)}_j)^2-2w_2x^{(k)}_j\\
&\quad +2a_{ij}(t_1t_2+x^{(k)}_i x^{(k)}_j).
\end{align*}
Since adding or subtracting a constant does not change the minimum point of a function, finding the minimum point of $h(t_1,t_2)$ is equivalent to finding the minimum point of $\tilde{h}(t_1,t_2)$ with
\begin{equation*}
\tilde{h}(t_1,t_2)=\left\{
\begin{aligned}
&\tilde{h}_1(t_1,t_2),~~\text{if}~t_1 \ge 0~\text{and}~t_2 \ge 0,\\
&\tilde{h}_2(t_1,t_2),~~\text{if}~t_1 \ge 0~\text{and}~t_2 < 0,\\
&\tilde{h}_3(t_1,t_2),~~\text{if}~t_1 < 0~\text{and}~t_2 \ge 0,\\
&\tilde{h}_4(t_1,t_2),~~\text{if}~t_1 < 0~\text{and}~t_2 < 0,\\
\end{aligned}
\right.
\end{equation*}
where
\begin{align*}
\tilde{h}_1(t_1,t_2)
&=(a_{ii}-1)t_1^2-2(x^{(k)}_i a_{ii}-w_1+a_{ij}x^{(k)}_j)t_1\\
&\quad +(a_{jj}-1)t_2^2-2(x^{(k)}_j a_{jj}-w_2+a_{ij}x^{(k)}_i)t_2+2a_{ij}t_1t_2,\\
\tilde{h}_2(t_1,t_2)
&=(a_{ii}-1)t_1^2-2(x^{(k)}_i a_{ii}-w_1+a_{ij}x^{(k)}_j)t_1\\
&\quad +(a_{jj}+1)t_2^2-2(x^{(k)}_j a_{jj}-w_2+a_{ij}x^{(k)}_i)t_2+2a_{ij}t_1t_2,\\
\tilde{h}_3(t_1,t_2)
&=(a_{ii}+1)t_1^2-2(x^{(k)}_i a_{ii}-w_1+a_{ij}x^{(k)}_j)t_1\\
&\quad +(a_{jj}-1)t_2^2-2(x^{(k)}_j a_{jj}-w_2+a_{ij}x^{(k)}_i)t_2+2a_{ij}t_1t_2,\\
\tilde{h}_4(t_1,t_2)
&=(a_{ii}+1)t_1^2-2(x^{(k)}_i a_{ii}-w_1+a_{ij}x^{(k)}_j)t_1\\
&\quad +(a_{jj}+1)t_2^2-2(x^{(k)}_j a_{jj}-w_2+a_{ij}x^{(k)}_i)t_2+2a_{ij}t_1t_2.
\end{align*}

Since $A-I$ is symmetric positive definite, the submatrix
$$
\begin{bmatrix}
a_{ii}& a_{ij}\\
a_{ji}& a_{jj}
\end{bmatrix}
$$
is also symmetric positive definite. Then it follows from Lemma \ref{lem:3} that
\begin{itemize}
  \item [(1)] If we extend the domain of $\tilde{h}_1(t_1,t_2)$ to $\mathbb{R}^2$, then $\tilde{h}_1(t_1,t_2)$ has the unique minimum point $[t_1,t_2]^\top$ with
      \begin{align}\label{eq:h1t1}
      t_1&=\quad\frac{-w_2 a_{ij}+w_1 a_{jj}-w_1+x^{(k)}_j a_{ij}+x^{(k)}_i(a_{ij}^2-a_{ii}a_{jj}+a_{ii})}{a_{ij}^2-(a_{ii}-1)(a_{jj}-1)},\\\label{eq:h1t2}
t_2&=\frac{-w_1 a_{ij}+w_2 a_{ii}-w_2+x^{(k)}_i a_{ij}+x^{(k)}_j(a_{ij}^2-a_{ii}a_{jj}+a_{jj})}{a_{ij}^2-(a_{ii}-1)(a_{jj}-1)};
\end{align}

  \item [(2)] If we extend the domain of $\tilde{h}_2(t_1,t_2)$ to $\mathbb{R}^2$, then $\tilde{h}_2(t_1,t_2)$ has the unique minimum point $[t_1,t_2]^\top$ with
      \begin{align}\label{eq:h2t1}
      t_1&=  \frac{-w_2 a_{ij}+w_1 a_{jj}+w_1-x^{(k)}_j a_{ij}+x^{(k)}_i(a_{ij}^2-a_{ii}a_{jj}-a_{ii})}{a_{ij}^2-(a_{ii}-1)(a_{jj}+1)}, \\\label{eq:h2t2}
      t_2& =\frac{-w_1 a_{ij}+w_2 a_{ii}-w_2+x^{(k)}_i a_{ij}+x^{(k)}_j(a_{ij}^2-a_{ii}a_{jj}+a_{jj})}{a_{ij}^2-(a_{ii}-1)(a_{jj}+1)};
      \end{align}

  \item [(3)] If we extend the domain of $\tilde{h}_3(t_1,t_2)$ to $\mathbb{R}^2$, then $\tilde{h}_3(t_1,t_2)$ has the unique minimum point $[t_1,t_2]^\top$ with
      \begin{align}\label{eq:h3t1}
      t_1&= \frac{-w_2 a_{ij}+w_1 a_{jj}-w_1+x^{(k)}_j a_{ij}+x^{(k)}_i(a_{ij}^2-a_{ii}a_{jj}+a_{ii})}{a_{ij}^2-(a_{ii}+1)(a_{jj}-1)},\\\label{eq:h3t2}
      t_2& =\frac{-w_1 a_{ij}+w_2 a_{ii}+w_2-x^{(k)}_i a_{ij}+x^{(k)}_j(a_{ij}^2-a_{ii}a_{jj}-a_{jj})}{a_{ij}^2-(a_{ii}+1)(a_{jj}-1)};
      \end{align}

  \item [(4)] If we extend the domain of $\tilde{h}_4(t_1,t_2)$ to $\mathbb{R}^2$, then $\tilde{h}_4(t_1,t_2)$ has the unique minimum point $[t_1,t_2]^\top$ with
      \begin{align}\label{eq:h4t1}
      t_1&= \frac{-w_2 a_{ij}+w_1 a_{jj}+w_1-x^{(k)}_j a_{ij}+x^{(k)}_i(a_{ij}^2-a_{ii}a_{jj}-a_{ii})}{a_{ij}^2-(a_{ii}+1)(a_{jj}+1)},\\\label{eq:h4t2}
      t_2& =\frac{-w_1 a_{ij}+w_2 a_{ii}+w_2-x^{(k)}_i a_{ij}+x^{(k)}_j(a_{ij}^2-a_{ii}a_{jj}-a_{jj})}{a_{ij}^2-(a_{ii}+1)(a_{jj}+1)};
      \end{align}

\item[(5)] One of the minimum points of $\tilde{h}_i(t_1,t_2)$ $(i=1,2,3,4)$ is the minimum point of $h(t_1,t_2)$ in $\mathbb{R}^2$.
\end{itemize}

Then the solution $(\alpha^{(k)}, \beta^{(k)})$ can be determined as follows.

If $t_1$ and $t_2$ defined as in \eqref{eq:h1t1} and \eqref{eq:h1t2} satisfy $t_1\ge 0$ and $t_2\ge 0$, then the minimum point of $\tilde{h}_1(t_1,t_2)$ is the unique minimum point of $h(t_1,t_2)$, which implies that
\begin{align*}
\alpha^{(k)} &=t_1-x^{(k)}_i=\frac{-w_2a_{ij}+w_1a_{jj}-w_1+x^{(k)}_ja_{ij}-x^{(k)}_ia_{jj}+x^{(k)}_i}{a_{ij}^2
-(a_{ii}-1)(a_{jj}-1)},\\
\beta^{(k)} &=t_2-x^{(k)}_j=\frac{-w_1a_{ij}+w_2a_{ii}-w_2+x^{(k)}_ia_{ij}-x^{(k)}_ja_{ii}+x^{(k)}_j}{a_{ij}^2-
(a_{ii}-1)(a_{jj}-1)}.
\end{align*}

If $t_1$ and $t_2$ defined as in \eqref{eq:h2t1} and \eqref{eq:h2t2} satisfy $t_1\ge 0$ and $t_2 < 0$, then the minimum point of $\tilde{h}_2(t_1,t_2)$ is the unique minimum point of $h(t_1,t_2)$, which implies that
\begin{align*}
\alpha^{(k)} &=\frac{-w_2a_{ij}+w_1a_{jj}+w_1-x^{(k)}_ja_{ij}-x^{(k)}_ia_{jj}
-x^{(k)}_i}{a_{ij}^2-(a_{ii}-1)(a_{jj}+1)},\\
\beta^{(k)} &=\frac{-w_1a_{ij}+w_2a_{ii}-w_2+x^{(k)}_ia_{ij}
+x^{(k)}_ja_{ii}-x^{(k)}_j}{a_{ij}^2-(a_{ii}-1)(a_{jj}+1)}.
\end{align*}

If $t_1$ and $t_2$ defined as in \eqref{eq:h3t1} and \eqref{eq:h3t2} satisfy $t_1< 0$ and $t_2 \ge 0$, then the minimum point of $\tilde{h}_3(t_1,t_2)$ is the unique minimum point of $h(t_1,t_2)$, which implies that
\begin{align*}
\alpha^{(k)} &=\frac{-w_2a_{ij}+w_1a_{jj}-w_1+x^{(k)}_ja_{ij}
+x^{(k)}_ia_{jj}-x^{(k)}_i}{a_{ij}^2-(a_{ii}+1)(a_{jj}-1)},\\
\beta^{(k)} &=\frac{-w_1a_{ij}+w_2a_{ii}+w_2-x^{(k)}_ia_{ij}
-x^{(k)}_ja_{ii}-x^{(k)}_j}{a_{ij}^2-(a_{ii}+1)(a_{jj}-1)}.
\end{align*}

If $t_1$ and $t_2$ defined as in \eqref{eq:h4t1} and \eqref{eq:h4t2} satisfy $t_1< 0$ and $t_2 < 0$, then the minimum point of $\tilde{h}_4(t_1,t_2)$ is the unique minimum point of $h(t_1,t_2)$, which implies that
\begin{align*}
\alpha^{(k)} &=\frac{-w_2a_{ij}+w_1a_{jj}+w_1-x^{(k)}_ja_{ij}+x^{(k)}_ia_{jj}
+x^{(k)}_i}{a_{ij}^2-(a_{ii}+1)(a_{jj}+1)},\\
\beta^{(k)} &=\frac{-w_1a_{ij}+w_2a_{ii}+w_2-x^{(k)}_ia_{ij}
+x^{(k)}_ja_{ii}+x^{(k)}_j}{a_{ij}^2-(a_{ii}+1)(a_{jj}+1)}.
\end{align*}

According to Lemma \ref{lem:3}, the determined step sizes $\alpha^{(k)}$ and $\beta^{(k)}$ imply  $f(x^{(k+1)})<f(x^{(k)})$. Therefore, Algorithm \ref{alg:2} is called as the monotone block coordinate descent method.

Before ending this section, we will analyze the convergence of Algorithm \ref{alg:2}.

%
%

For the given $x^{(0)}\in\mathbb{R}^n$, define the level set
\begin{equation}\label{eq:level}
\mathcal{L}(x^{(0)})=\{x\in\mathbb{R}^n:f(x)\leq f(x^{(0)})\}.
\end{equation}

\begin{lem}\label{lem:5}
If $A-I\in \mathbb{R}^{n\times n}$ is a symmetric positive definite matrix and function $f$ is defined by \eqref{eq:f}, then the level set \eqref{eq:level} is a compact set.
\end{lem}
\begin{proof}
It is known continuously by the function $f$ that $\mathcal{L}(x^{(0)})$ is closed, the following proof shows that the level set $\mathcal{L}(x^{(0)})$ is bounded.

Assume, for contradiction, that $\{z^{(k)}\}\subseteq \mathcal{L}(x^{(0)})$ is an unbounded sequence. Then we have
$$
f(z^{(k)})\leq f(x^{(0)}),
$$
which implies
$$
(z^{(k)})^\top A z^{(k)}-(z^{(k)})^\top |z^{(k)}|-2(z^{(k)})^\top b\leq f(x^{(0)}).
$$
Assume $z^*$ is a accumulation point of the sequence $\{\frac{z^{(k)}}{\|z^{(k)}\|}\}$, then
$$
\frac{(z^{(k)})^\top A z^{(k)}}{\|z^{(k)}\|^2}-\frac{(z^{(k)})^\top |z^{(k)}|}{\|z^{(k)}\|^2}-\frac{2(z^{(k)})^\top b}{\|z^{(k)}\|^2} \leq \frac{f(x^{(0)})}{\|z^{(k)}\|^2}.
$$
Taking the limit as $k\rightarrow \infty$, we obtain
\begin{equation}\label{eq:z}
(z^*)^\top (A-\mathcal{D}(z^*))z^*\leq 0.
\end{equation}
Since $A-I$ is positive definite, $A-\mathcal{D}(z^*)$ is positive definite. Combining this with \eqref{eq:z}, we conclude $z^*=0$, which contradicts with $\|z^*\|=1.$

To sum up, the level set $\mathcal{L}(x^{(0)})$ is closed and bounded, so the level set $\mathcal{L}(x^{(0)})$ is compact.
\end{proof}

Then the following convergence result is obtained by Theorem~\ref{thm:1}, Lemma \ref{lem:5} and
\cite[Proposition 6]{grsc2000}.

\begin{thm}
Assume that $A - I$ is symmetric positive definite. Then the sequence $\{x^{(k)}\}$ generated by Algorithm \ref{alg:2} has a limit point which is the solutions of AVE~\eqref{eq:ave}.
\end{thm}

\section{Numerical results}\label{sec:Num}
In this section, three numerical examples are used to illustrate the effectiveness of Algorithm 2. In both examples, the performance of Algorithm 2 (denoted by BCDA) and Algorithm 1 (denoted by MGSM) are compared. In the numerical results, we will report ``IT'' (the number of iterations. For Algorithm \ref{alg:1}, IT $=k\times n$.), ``Time'' (the elapsed CPU time in seconds) and ``RES'' which is defined by
$$
{\rm RES} =\frac{\|b+|x^{(k)}|-Ax^{(k)}\|}{\|b\|}.
$$
All tests are started from initial zero vector (except Example~\ref{exam:4.3}) and terminated if the current iteration satisfies ${\rm RES} \le 10^{-6}$. In our computations, all runs are implemented in MATLAB (version 9.10 (R2021a)) on a personal computer with IntelCore(TM) i7 CPU 2.60 GHz, 16.0GB memory.

In order to visualize the performance of the two algorithms, a two-dimensional problem is first considered.

\begin{exam}\label{exam:4.1}
Consider AVE \eqref{eq:ave}, where
\begin{equation*}
A=\begin{bmatrix}
\frac{3}{2}&\frac{1}{4}\\
\frac{1}{4}&\frac{3}{2}
\end{bmatrix}
,~b=[\frac{1}{4},1]^\top\in\mathbb{R}^2.
\end{equation*}

The numerical results are shown in Table 2, and the iterative trajectories of the algorithms are  shown in Figure~\ref{fig1}, which clearly shows that BCDA is monotonic while MGSM is not.

\begin{table}[H]
\centering
\caption{Numerical results for Examples~\ref{exam:4.1}.}\label{table2}
\begin{tabular}{lllllllll}\hline
&$n=2$     &BCDA    &MGSM         \\\hline
&IT      &1       &2            \\
&Time    &0.003460       &0.010728  \\
&RES     &5.3854e-17  &4.8168e-16   \\\hline
\end{tabular}
\end{table}

\begin{figure}[h]
  \centering
  \includegraphics[width=0.7\linewidth]{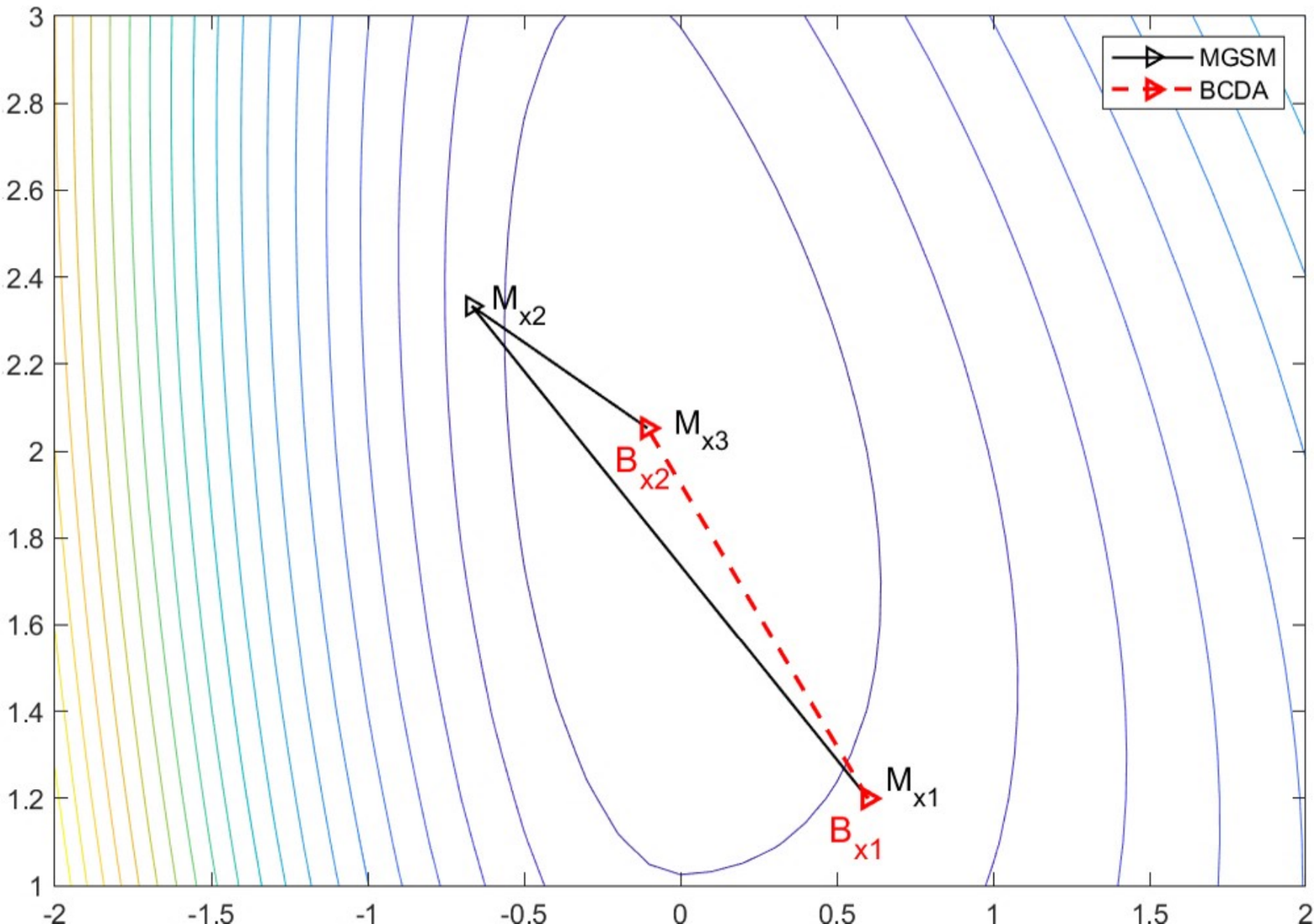}
  \caption{Iterative trajectory of the algorithm.}\label{fig1}
\end{figure}

\end{exam}

Next, a high-dimensional example is used to demonstrate the performance of the proposed algorithm.

\begin{exam}\label{exam:4.2}
Consider AVE \eqref{eq:ave}, where
\begin{eqnarray*}
A=tridiag(\frac{3}{4},4,\frac{3}{4})=
\left(\begin{array}{cccccc}
4&\frac{3}{4}&0&\cdots&0&0\\
\frac{3}{4}&4&\frac{3}{4}&\cdots&0&0\\
0&\frac{3}{4}&4&\cdots&0&0\\
\vdots&\vdots&\vdots&\vdots&\vdots&\vdots\\
0&0&0&\cdots&4&\frac{3}{4}\\
0&0&0&\cdots&\frac{3}{4}&4\\
\end{array}\right)\in\mathbb{R}^{n\times n}
\end{eqnarray*}
and ~$b=[\frac{1}{2},1,\frac{1}{2},1,\ldots,\frac{1}{2},1]^\top\in\mathbb{R}^n$. The numerical results are shown in Table 3. For this example, Table 3 shows that BCDA outperforms MGSM in both the number of iteration steps and the iteration time.

\begin{table}[H]
\centering
\caption{Numerical results for Examples~\ref{exam:4.2}.}\label{table3}
\begin{tabular}{lllllllll}\hline
&Method     &$n$    &1000         &1500        &2000         &2500        &3000         \\\hline
&BCDA     &IT     &2000            &2250           &3000            &3750           &4500     \\
&{}       &Time   &0.0187       &0.0231      &0.0589      &0.1357      &0.2225           \\
&{}       &RES    &9.1891e-08   &8.8970e-07  &7.7050e-07   &6.8916e-07  &6.2911e-07  \\
&MGSM     &IT     &6000            &9000           &12000            &15000           &18000     \\
&{}       &Time   &0.0377       &0.0561      &0.1796       &0.4969      &0.8426           \\
&{}       &RES    &3.0374e-07   &3.0432e-07  &3.0461e-07   &3.0478e-07  &3.0490e-07   \\\hline
\end{tabular}
\end{table}
\end{exam}

\begin{exam}\label{exam:4.3}
Consider AVE \eqref{eq:ave}, where
$$
A=\begin{bmatrix}
    1 & 0.25  \\
    0.25 & 1
  \end{bmatrix},
b=\begin{bmatrix}
    1 \\
    0.5
  \end{bmatrix}.
$$
In this case, $A-I$ is not symmetric positive definite. It can be verified that this AVE has a solution $x^*=[2,4]^\top$. For both algorithms  MGSM and BCDA, the initial point $x^{(0)}=[0.1,-1]^\top$ is selected. The relative residual curves of the algorithms are shown in Figure \ref{fig2}. In this case, the algorithm  MGSM does not converge. In fact, for the algorithm  MGSM there is
\begin{equation*}
x^{(1)}=\begin{bmatrix}
          0.1 \\
          -1
        \end{bmatrix},~
x^{(2)}=\begin{bmatrix}
          -30 \\
          4
        \end{bmatrix},~
x^{(3)}=\begin{bmatrix}
          2 \\
          -12
        \end{bmatrix},~
x^{(4)}=\begin{bmatrix}
          -30 \\
          4
        \end{bmatrix},\dots,
\end{equation*}
that is, the algorithm MGSM iterates between points $[-30,4]^\top$ and $[2,-12]^\top$ starting at $x^{(2)}$. However, the algorithm BCDA can converge to the solution $x^*$.
\end{exam}

\begin{figure}[H]
  \centering
  \includegraphics[width=0.7\linewidth]{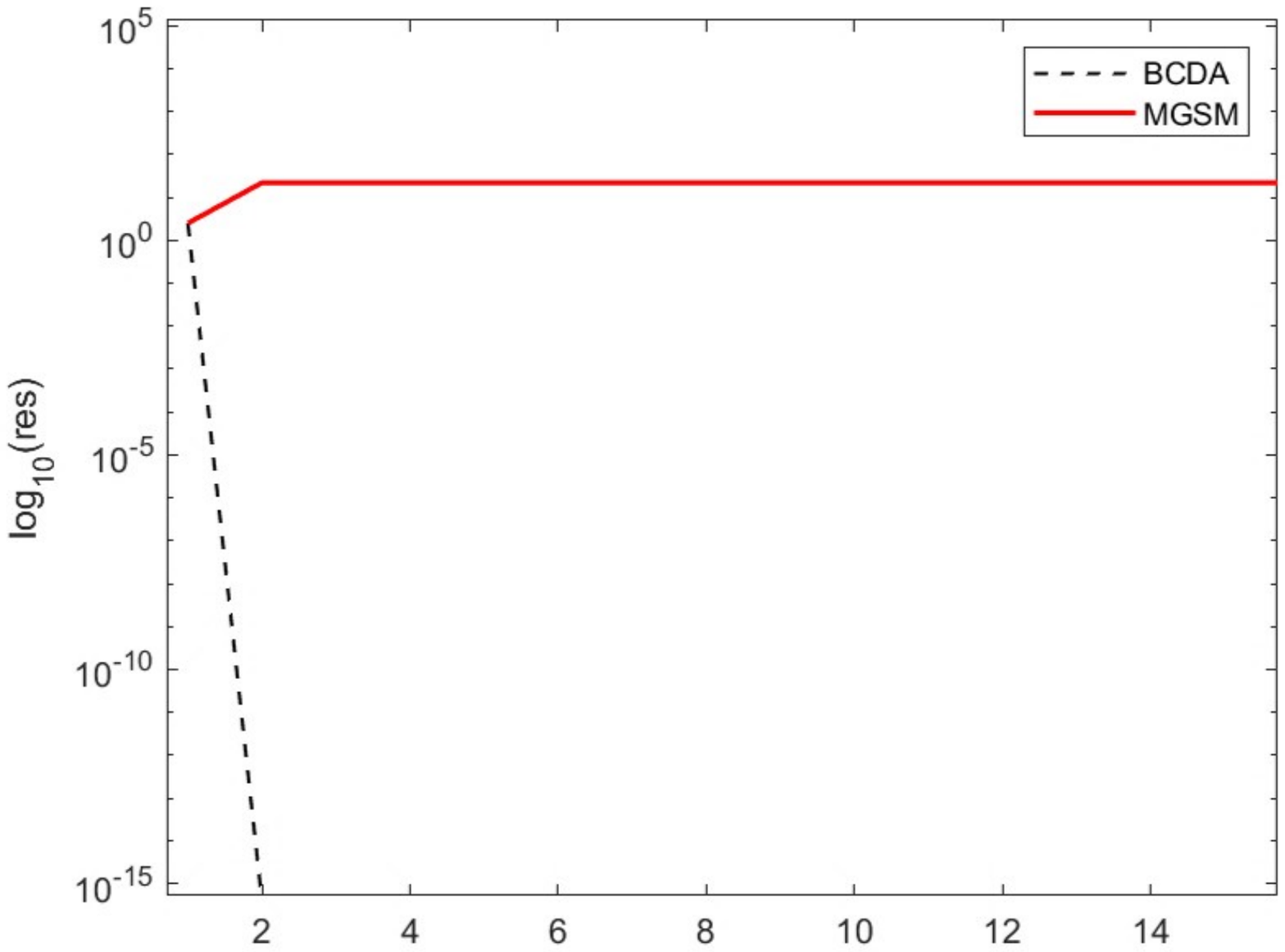}
  \caption{Plot of relative residuals for Example \ref{exam:4.3}.}\label{fig2}
\end{figure}

\section{Conclusion}\label{sec:conclusion}
A monotone block coordinate descent method is proposed to solve AVE \eqref{eq:ave}, which can be  superior to the method presented in \cite{nika2011}. A future study direction is to develop a nonmonotone block coordinate descent method with convergence guarantee to solve AVE \eqref{eq:ave}.

\end{document}